\documentclass[reqno,a4paper,12pt]{amsart}
\usepackage{amssymb}
\usepackage{graphicx, color, hyperref,subcaption}
\usepackage[utf8]{inputenc}

\allowdisplaybreaks

\newtheorem{theorem}{Theorem}

\newtheorem{proposition}{Proposition}

\theoremstyle{definition}
\newtheorem{definition}{Definition}

\theoremstyle{remark}
\newtheorem{remark}{Remark}

\newcommand{\Z}{\mathbb{Z}}
\newcommand{\C}{\mathbb{C}}
\newcommand{\N}{\mathbb{N}}
\newcommand{\R}{\mathbb{R}}
\renewcommand{\S}{\mathbb{S}}

\newcommand{\tr}{\operatorname{tr}}

\begin{document}

\title[Expected energy on flat tori]{Expected Riesz energy of some 
determinantal processes on flat tori}

\author{Jordi Marzo and Joaquim Ortega-Cerd\`a}

\address{Departament de Matem\`atiques i Inform\`atica, Universitat de 
Barcelona \& Bar\-ce\-lo\-na Graduate School of Mathematics, Gran Via 585, 
08007, Barcelona, Spain}
\email{jmarzo@ub.edu}

\address{Departament de Matem\`atiques i Inform\`atica, Universitat de 
Barcelona \& Bar\-ce\-lo\-na Graduate School of Mathematics, Gran Via 585, 
08007, Barcelona, Spain}
\email{jortega@ub.edu}

\keywords{Riesz energy, Epstein zeta function, Determinantal processes, Torus, 
Rearrangement inequality} 
\thanks{This research has been partially supported by the MTM2014-51834-P grant 
by the Ministerio de Econom\'{\i}a y Competitividad, Gobierno de Espa\~na and by 
the Ge\-ne\-ra\-li\-tat de Catalunya (project 2014 SGR 289).}

\begin{abstract} We compute the expected Riesz energy of random points on flat 
tori drawn from certain translation invariant determinantal processes and 
determine the process in the family providing the optimal asymptotic expected 
Riesz energy. 
\end{abstract}

\maketitle

\section{Introduction}

Our objective is to study the asymptotics of the expected Riesz energy of 
certain point processes (random finite point configurations) in a flat torus 
$\Omega\subset \R^d$. If $\Lambda$ is a lattice in $\R^d$ (i.e. $\Lambda=A\Z^d$ 
for some nonsingular square matrix $A$) we identify the fundamental domain 
$$\Omega=\{ t_1 v_1+\dots +t_d v_d : t_1,\dots, t_d \in [0,1)\},$$ where the 
column vectors of $A=[v_1,\dots , v_d]$ is a $\Z$ basis of the lattice 
$\Lambda,$ with the flat torus $\R^d / \Lambda$.

For the sphere $\S^2,$ the authors in \cite{ASZ14} estimate asymptotically the 
expected energy of points of the, so-called, spherical ensemble. In 
\cite{BMOC16}, the authors study the harmonic ensemble in $\S^d$ and prove, in 
some cases, the optimality of the expected asymptotic energy of this process 
among rotation invariant determinantal processes. In both cases, the 
expected asymptotic energy was used to get upper bounds for the minimal Riesz 
energy. Here, we study also the optimality of the expected asymptotic energy 
among a collection of determinantal processes invariant under translations and 
it turns out that the best process can be found as an easy consequence of 
Riesz's rearrangement inequality.

This provides explicit examples with the lowest energy bounds known on the 
torus in high dimensions. 
\subsection{Riesz energy}

To define the Riesz energy in this periodic setting we follow 
\cite{HSSS14,HSSS15}, see also \cite[Section 9]{CK07}. Given a lattice 
$\Lambda=A\Z^d\subset \R^d$ the Epstein Hurwitz zeta function for $\Lambda$ is 
defined, for $s>d,$ as $$\zeta_\Lambda(s;x)=\sum_{v\in \Lambda} 
\frac{1}{|x+v|^s},\;\;\;x\in \R^d.$$ Observe that $\zeta_\Lambda(s;x)$ is the 
$\Lambda$-periodic potential generated by the Riesz s-energy $|x|^{-s}$. 

When $s\le d$ the sum above is infinite for all $x\in \R^d$. For fixed $x\in 
\R^d\setminus \Lambda$ define the function
\begin{equation}		\label{analytic}
F_{s,\Lambda}(x)=\sum_{v\in \Lambda}\int_1^{+\infty} e^{-|x+v|^2 
t}\frac{t^{\frac{s}{2}-1}}{\Gamma\left(  \frac{s}{2} \right)}dt+
\frac{1}{|\Lambda|}\sum_{w\in \Lambda^*\setminus \{0\}} 
e^{2\pi i \langle x,w \rangle} \int_0^1 
\frac{\pi^{d/2}}{t^{d/2}}e^{-\frac{\pi^2 |w|^2}{t}} 
\frac{t^{\frac{s}{2}-1}}{\Gamma\left(  \frac{s}{2} \right)}dt, 
\end{equation}
where $\Lambda^*=\{ x\in \R^d :\forall \lambda \in \Lambda\;\; \langle 
x,\lambda \rangle \in \Z \}=(A^t)^{-1}\Z^d$ is the 
dual lattice and $|\Lambda|=|\det A|$ is the co-volume of $\Lambda$.

Then, $F_{s,\Lambda}(x)$ is an entire function of $s$ and therefore by the 
relation 
$$F_{s,\Lambda}(x)=\zeta_\Lambda(s;x)+\frac{2\pi^{d/2}|\Lambda|^{-1}}{
\Gamma\left(  \frac{s}{2} \right)(d-s)},\;\;s>d,$$ we obtain a analytic 
continuation of $\zeta_\Lambda(s;x)$ to $s\in \C\setminus \{ d\}$. Observe that 
the function $1/\Gamma(s)$ is entire and that all the sums in (\ref{analytic}) 
converge uniformly. We are interested in the range $0<s<d$.

For $\omega=(x_1,\dots , x_N) \in \Omega^N$ define, for $0<s<d,$ the periodic 
Riesz $s$-energy of $\omega$ by $$E_{s,\Lambda}(\omega)=\sum_{k\neq j} 
F_{s,\Lambda}(x_k-x_j),$$ and the minimal periodic Riesz $s$-energy by 
$$\mathcal{E}_{s,\Lambda}(N)=\inf_{\omega \in (\R^d)^N} E_{s,\Lambda}(\omega).$$

\subsection{Determinantal processes}

For the introductory background we follow \cite[Chap.~4]{GAF}. 

We denote as $\mathcal{X}$ a (simple) random point process in a compact set 
$\Omega\subset \R^d$. And let $\mu$ be the normalized Lebesgue measure. A way 
to 
describe the process is to specify the random variable counting the number of 
points of the process in $D,$ for all Borel sets $D\subset \Omega$. We denote 
this random variable as $\mathcal{X}(D)$.

These point processes are characterized by their joint intensity functions 
$\rho_k$ 
in $\Omega^k$ satisfying
that 
\[
\mathbb{E}\left[ \mathcal{X}(D_1)\cdots \mathcal{X}(D_k) \right] 
=\int_{D_1\times \dots \times D_k} \rho_k(x_1,\dots, x_k) d\mu(x_1)\dots 
d\mu(x_k),
\]
for any family of mutually disjoint subsets $D_1,\dots ,D_k \subset \Omega$. 
We assume that $\rho_k(x_1,\dots, x_k)=0$ 
when $x_i=x_j$ for $i\neq j$.

A random point process is called determinantal with kernel 
$K:\Omega\times \Omega \rightarrow \C,$ if it is simple and the joint 
intensities
with respect to a background measure $\mu$ are given by
\[
\rho_k(x_1,\dots, x_k)=\det (K(x_i,x_j))_{1\le i,j\le k}, 
\]
for every $k\ge 1$ and $x_1,\dots, x_k \in \Omega$.

To define the processes we will consider only projection kernels.

\begin{definition}
 We say that $K$ is a projection kernel if it is a Hermitian projection kernel, 
i.e. the integral operator in $L^2(\mu)$ with kernel $K$ is self-adjoint and 
has 
eigenvalues $1$ and $0$.
\end{definition}

By Macchi-Soshnikov's theorem
\cite[Theorem 4.5.5]{GAF},
a projection kernel $K(x,y)$ defines a determinantal process and it has 
$N$ points almost surely if the trace for the corresponding integral operator 
equals $N$, i.e.\ if
\[
 \int_{\Omega}K (x,x)d\mu(x)=N.
\]
Observe that the random vector in $\Omega^N$ generated with density 
$$\frac{1}{N!}\det 
(K(x_i,x_j))_{1\le i,j\le k},$$
is a determinantal process with the right 
marginals i.e. the joint intensities are given by determinants of the kernel 
\cite[Remark~4.2.6]{AGZ}.

Given now a function $f:\Omega\times\Omega\rightarrow[0,\infty)$ 
it is easy to compute the 
expected pair potential energy, \cite[Formula (1.2.2)]{GAF}:

\begin{proposition}\label{th:determinantal}
Let $K(x,y)$ be a projection kernel with trace $N$ in $\Omega$
and let 
$\omega=(x_1,\ldots,x_N)\in \Omega^N$ be $N$ random points generated by the 
corresponding determinantal point process. Then, for any 
measurable $f:\Omega\times\Omega\rightarrow[0,\infty)$ we have
\[
 \mathbb{E}_{\omega\in \Omega^N}\left(\sum_{i\neq j}f(x_i,x_j)\right)= 
\int_{x,y\in\Omega}\left(K(x,x)K(y,y)-|K(x,y)|^2\right)f(x,y)\,
d\mu (x)\,d\mu (y).
\]
\end{proposition}

\subsubsection{Flat torus}

In our setting we take as $\Omega\subset \R^d$ the flat torus $\R^d/\Lambda,$ 
for some lattice $\Lambda$ with dual $\Lambda^*$.

To construct the kernel we consider for $w\in \Lambda^*$, the Laplace-Beltrami 
eigenfunctions $f_w(u)=e^{2\pi i \langle u, w \rangle }$ of eigenvalue $-4 
\pi^2 
\langle w, w\rangle$. Then $$\Delta f_w + 4 \pi^2 \langle w, w\rangle f_w=0,$$ 
and $\{ f_w \}_{w\in \Lambda^*}$ are orthonormal in $L^2(\Omega),$ with respect 
to the normalized Lebesgue measure $\mu$ in $\Omega$ $$\int_\Omega f_w(u) 
\overline{f_{w'}(u)} d\mu(u)=\delta_{w,w'}$$ for $w,w'\in \Lambda^*$.

Now, we consider functions $\kappa=(\kappa_N)_{N\ge 0}$ where each 
$\kappa_N:\Lambda^* \longrightarrow \{0,1\}$ has compact support and we define 
the kernels $$K_N(u,v)=\sum_{w\in \Lambda^*} \kappa_N(w) e^{2\pi i \langle u-v, 
w \rangle },\;\;\;u,v\in \Omega,$$ and the corresponding determinantal point 
processes on the flat torus $\Omega$. For these processes (we are not going to 
distinguish between the integral operator defined by the kernel $K_N$ and the 
kernel itself) we get $$\mbox{tr}(K_N)=\int_\Omega K_N(u,u) d\mu(u)=\sum_{w\in 
\Lambda^*} \kappa_N(w)=\# \,\mbox{supp}\, \kappa_N,$$ points almost surely.

\subsubsection{Examples}

Let $\Lambda=\Z^d$. The flat torus $\mathbb{T}^d=\R^d/\Z^d$ has 
$\Omega=[0,1)^d$.

Consider
$$\kappa_N(x)=\chi_{[-N,N]^d}(x),\;\; x\in \R^d,$$
then for $u,v\in [0,1)^d$
$$K_N(u,v)=\sum_{\lambda\in \Z^d,\| \lambda\|_{\infty}\le N} e^{2\pi i \langle 
u-v,\lambda \rangle}=
\prod_{j=1}^d D_N(u_j-v_j),
$$
where 
$$D_N(x)=\frac{\sin[\left( N+\frac{1}{2}\right)x]}{\sin \left( 
\frac{x}{2}\right)},\;\;x\in \mathbb{T},$$
is the Dirichlet kernel.

If $N\in \N$ can be expressed as a sum $N=\lambda_1^2+\cdots +\lambda_d^2$ for 
$\lambda \in \Z^d$. We can define
$$K_N(u,v)=\sum_{\lambda\in \Z^d,\| \lambda\|_{2}^2= N} e^{2\pi i \langle 
u-v,\lambda \rangle},\;\;\;u,v\in[0,1)^d.
$$

Observe that $\{ e^{2\pi i \langle \cdot,\lambda \rangle} \}_{\lambda\in 
\Z^d,\| \lambda\|^2_2=N}$ span the eigenspace corresponding to
eigenvalue $-4\pi^2 N$ and it has dimension $r_d(N)$. Where $r_d(N)$ is the 
number of different ways that $N$ may be expressed as a sum of $d$ 
squares (the order in the sum of the squares is counted as distinct). For 
example $r_2(4)=4$ because
$4=0+2^2=0+(-2)^2=2^2+0=(-2)^2+0$. This corresponds to the four distinct points 
$(0,\pm 2)$ and $(\pm 2,0)$.

\subsection{Some known results about minimal periodic Riesz $s$-energy}

It was shown in \cite{HSSS15} that for $0<s<d$ there exists 
a constant $C_{s,d}$ independent of $\Lambda$ such that for $N\to \infty$
\begin{equation}\label{asymptotics}
 \mathcal{E}_{s,\Lambda}(N)=\frac{2\pi^{d/2}|\Lambda|^{-1}}{\Gamma\left(  
\frac{s}{2} \right)(d-s)} N^2+C_{s,d} |\Lambda|^{-s/d} N^{1+\frac{s}{d}}+
o(N^{1+\frac{s}{d}}).
\end{equation}
The constant $C_{s,d}$ above is not known (unless $d=1$). In \cite{HSSS15} the 
authors found an upper bound in terms of the Epstein zeta function.
Recall that for a lattice $\Lambda\subset \R^d,$ the Epstein zeta function 
$\zeta_\Lambda(s)$ defined by
$$\zeta_\Lambda(s)=\sum_{v\in \Lambda\setminus \{0\}} 
\frac{1}{|v|^s},\;\;\;s>d,$$
can be extended analytically (as in (\ref{analytic})) to $\C\setminus \{d\}$.
One can see easily that $\zeta_\Lambda(0)=-1$ and the residue of 
$\zeta_\Lambda(s)$ in $d$ 
is $2\pi^{d/2}/\Gamma(d/2)=\omega_{d-1}$. The result in \cite[Corollary 
3]{HSSS15} is that for $0<s<d$ 
$$C_{s,d}\le \inf_{\Lambda} \zeta_\Lambda(s),$$
where $\Lambda$ runs on the lattices with $|\Lambda|=1$. 
It has been conjectured (see \cite{BHS12}) that if $d=2,4,8$, or $24$, 
then	 $C_{s,d}= \zeta_{\Lambda_d}(s)$
where $\Lambda_d$ denotes (respectively) the hexagonal lattice, the $D_4$ 
lattice, the $E_8$ lattice
and the Leech lattice (scaled to have $|\Lambda_d|=1$).
When $d=1$ indeed $C_{s,1}=\zeta_\Z(s)=2\zeta(s)$. 
For $d=2$ it is known, due to the work of several authors, that $\inf_{\Lambda} 
\zeta_\Lambda(s)$
is attained for the triangular lattice, see \cite{Mon88} where the result is 
deduced from the corresponding result for theta functions.
It is observed in \cite{SS06} that
from Siegel's integration formula it follows that
$$\int  \zeta_\Lambda(s)d\lambda_d (\Lambda)=0,$$
where $d\lambda_d$ is the volume measure in the space of lattices, \cite[p. 172]{Ter88}. One deduces 
then that $C_{s,d}<0,$ although for large dimensions there are no examples 
providing negative bounds. Indeed, from \cite{SS06}, see also \cite[Theorem 
1]{Ter80}, all explicitly known lattices in large dimensions 
are such that the corresponding Epstein zeta functions
have a zero 
in $0<s<d,$ i.e. the analogue of the Riemann hypothesis fails for Epstein zeta 
functions, see Remark \ref{remark2}.

\section{Expected energies}

By Proposition \ref{th:determinantal} the expected periodic Riesz $s$-energy of 
$t_N=\tr (K_N)$ random points $\omega=(x_1,\dots,x_{t_N})$ drawn from the 
determinantal process defined by the kernel $K_N(u,v)$
is
$$\mathbb{E}_{\omega \in 
(\R^d)^{t_N}}(E_{s,\Lambda}(\omega))=\int_{\Omega^2}(K_N(u,u)K_N(v,v)-|K_N(u,
v)|^2) F_{s,\Lambda}(u-v)d\mu(u)d\mu(v).$$ 

It is easy to see, \cite{HSSS14}, that for $0<s<d$
$$\int_{\Omega^2}  
F_{s,\Lambda}(u-v)d\mu(u)d\mu(v)=\frac{2\pi^{d/2}|\Lambda|^{-1}}{\Gamma\left(  
\frac{s}{2} \right)(d-s)},$$
and therefore by translation invariance
\begin{equation} 		\label{energy_expr1}
 \mathbb{E}_{\omega\in (\R^d)^{t_N}}(E_{s,\Lambda}(\omega))=  
\frac{2\pi^{d/2}|\Lambda|^{-1}}{\Gamma\left(  \frac{s}{2} \right)(d-s)} {t_N}^2 
-
\int_{\Omega} |K_N(u,0)|^2 F_{s,\Lambda}(u) d\mu(u).
\end{equation}

Our first result is a nice closed expression for the integral above.

\begin{theorem}				\label{explicit}
 Let $\omega=(x_1,\ldots,x_{t_N})$ be drawn from the determinantal process on 
the flat torus $\R^d / \Lambda$
 given by the kernel 
 $$K_N(u,v)=\sum_{w\in \Lambda^*} \kappa_N(w) e^{2\pi i \langle u-v, w \rangle 
},$$
 with $\kappa_N(w) \in \{0,1\}$ for $w\in \Lambda^*$ and
 $\sum_{w\in \Lambda^*} \kappa_N(w)=t_N$.

Then, for $0<s<d$,
$$\mathbb{E}_{\omega \in (\R^d)^{t_N}}(E_{s,\Lambda}(\omega))
=\frac{2\pi^{d/2}}{\Gamma\left(  \frac{s}{2} \right)(d-s) |\Lambda|} 
({t_N}^2-{t_N})
-
\frac{\pi^{s-\frac{d}{2}} \Gamma\left(  \frac{d-s}{2} \right) }{ \Gamma\left(  
\frac{s}{2} \right) |\Lambda|}\sum_{\substack{w,w' \in \Lambda^* \\ w\neq w'}} 
\frac{\kappa_N(w)\kappa_N(w')}{|w-w'|^{d-s}}.$$
\end{theorem}

\begin{remark}
Observe that for $N$ random points chosen independently and uniformly in 
$\Omega$ (i.e. for the Poisson point process) the expected energy is given by
$$\mathbb{E}_{\mbox{uniform}}(E_{s,\Lambda})=\frac{2\pi^{d/2}|\Lambda|^{-1}}{
\Gamma\left(  \frac{s}{2} \right)(d-s)} (N^2-N),$$
so the improvement (lowering) in the determinantal case comes from the last 
summand above.
Therefore, to get a good upper bound for the minimal energy we 
want to maximize the sum $$\sum_{\substack{w,w' \in \Lambda^* \\ w\neq w'}} 
\frac{\kappa_N(w)\kappa_N(w')}{|w-w'|^{d-s}},\;\;\;\mbox{given }\;\;\sum_{w\in 
\Lambda^*}\kappa_N(w)=t_N.$$
This is not an easy task in general. For example, when $t_N=2$ this would lead 
to find the shortest non-zero vector in 
the lattice $\Lambda^*$ i.e.
$$m(\Lambda^*)=\min \{ |w|\;:\; w\in \Lambda^*\setminus \{ 0 \} \},$$
or equivalently, the density of the densest lattice sphere packing.
\end{remark}

\begin{proof}
To compute the integral in (\ref{energy_expr1}) we write
$$|K_N(u,0)|^2=\sum_{w,w' \in \Lambda^*} \kappa_N(w)\kappa_N(w') e^{2\pi i 
\langle u,w-w'  \rangle},$$ 
and using the expression for $F_{s,\Lambda}(u),$ where the sums
converge uniformly,
we get
$$\sum_{v\in \Lambda}\int_\Omega e^{-|u+v|^2 t}e^{2\pi i \langle u,w-w' 
\rangle}du=\int_{\R^d} e^{-|u|^2 t}e^{2\pi i \langle u,w-w' \rangle}du
=\left( \frac{\pi}{t} \right)^{d/2} e^{-\pi^2\frac{|w-w'|^2}{t}},$$
and
$$\frac{1}{|\Lambda|}\sum_{w,w' \in \Lambda^*} 
\kappa_N(w)\kappa_N(w')\sum_{\eta \in \Lambda^*\setminus \{0\}}
\int_0^1 \left( \frac{\pi}{t} \right)^{d/2} 
e^{-\pi^2\frac{|\eta|^2}{t}}\frac{t^{\frac{s}{2}-1}}{\Gamma\left( \frac{s}{2} 
\right)}
\left[ \int_\Omega e^{2\pi i \langle  u,w-w'+\eta \rangle } d\mu(u)\right] dt.$$
Observe that
$$\int_\Omega e^{2\pi i \langle  u,w-w'+\eta \rangle } 
d\mu(u)=\delta_{w'-w,\eta},$$
so
$$\frac{1}{|\Lambda|}\sum_{\substack{w,w' \in \Lambda^* \\ w\neq w'}} 
\kappa_N(w)\kappa_N(w')
\int_0^1 \left( \frac{\pi}{t} \right)^{d/2} 
e^{-\pi^2\frac{|w-w'|^2}{t}}\frac{t^{\frac{s}{2}-1}}{\Gamma\left( \frac{s}{2} 
\right)}dt.$$

Putting all together, and using that for $s<d$
$$\int_1^{+\infty} \left( \frac{\pi}{t} \right)^{d/2} t^{\frac{s}{2}-1}dt
=\frac{2 \pi^{d/2} }{d-s},$$
we get
$$\int_{\Omega} |K_N(u,0)|^2 F_{s,\Lambda}(u) 
d\mu(u)=\frac{1}{|\Lambda|}\sum_{w,w' \in \Lambda^* } \kappa_N(w)\kappa_N(w')
\int_1^{+\infty} \left( \frac{\pi}{t} \right)^{d/2} 
e^{-\pi^2\frac{|w-w'|^2}{t}}\frac{t^{\frac{s}{2}-1}}{\Gamma\left( \frac{s}{2} 
\right)}dt$$
$$+\frac{1}{|\Lambda|}\sum_{\substack{w,w' \in \Lambda^* \\ w\neq w'}} 
\kappa_N(w)\kappa_N(w')
\int_0^1 \left( \frac{\pi}{t} \right)^{d/2} 
e^{-\pi^2\frac{|w-w'|^2}{t}}\frac{t^{\frac{s}{2}-1}}{\Gamma\left( \frac{s}{2} 
\right)}dt$$
$$=\frac{2 \pi^{d/2} t_N }{(d-s)\Gamma\left( \frac{s}{2} 
\right)|\Lambda|}+\frac{1}{|\Lambda|}\sum_{\substack{w,w' \in \Lambda^* \\ 
w\neq w'}} \kappa_N(w)\kappa_N(w')
\left[ 
\int_0^{+\infty} \left( \frac{\pi}{t} \right)^{d/2} 
e^{-\pi^2\frac{|w-w'|^2}{t}}\frac{t^{\frac{s}{2}-1}}{\Gamma\left( \frac{s}{2} 
\right)}dt\right].
$$

This last integral converges for all $s<d,$ and using that (for $\alpha<-1$)

$$\int_0^{+\infty} e^{-c^2/t}t^{\alpha} dt=c^{2\alpha+2}\Gamma(-\alpha-1)$$

we get the result.
\end{proof}

Now we define a way to get different invariant kernels (by choosing different 
sequences of functions $\kappa$) and we estimate the 
corresponding expected energies.

\begin{definition}
 Let $\mathcal{D}\subset \R^d$ be an open bounded subset with boundary of 
measure zero and let $\Lambda\subset \R^d$ be a lattice. 
Define for $N\in \N$
the functions $\kappa^{\mathcal{D},\Lambda}=(\kappa_N)_{N\ge 0}$ where
$$
\kappa_N(w)=
\left\{
  \begin{array}{ll}
    1  & \mbox{if } w\in \Lambda^*\cap N^{1/d}\mathcal{D}, \\
    0 & \mbox{otherwise}. 
  \end{array}
\right.
$$ 
\end{definition}

\begin{proposition}
  Let $\Lambda\subset \R^d$ be a lattice and $\mathcal{D}\subset \R^d$ be an 
open bounded subset with boundary of measure zero
and such that $|\Lambda||\mathcal{D}|=1$.
 Let $\kappa^{\mathcal{D},\Lambda}=(\kappa_N)_{N\ge 0}$ be defined as above. 
Suppose that the trace of the corresponding kernel equals $t_N$ i.e.
$$\sum_{w\in \Lambda^*}\kappa_N(w)=t_N.$$ 
Then, for $0<s<d,$ if $\omega=(x_1, \dots ,x_{t_N})\in \Omega^{t_N}$
are $t_N$ points drawn from the determinantal process defined by 
$\kappa^{\mathcal{D},\Lambda}$
$$\mathbb{E}_{\omega\in (\R^d)^{t_N}}(E_{s,\Lambda}(\omega))
=\frac{2\pi^{d/2}|\Lambda|^{-1}}{\Gamma\left(  \frac{s}{2} \right)(d-s)} 
{t_N}^2 -
\frac{\pi^{s-\frac{d}{2}} \Gamma\left(  \frac{d-s}{2} \right) }{ \Gamma\left(  
\frac{s}{2} \right) |\Lambda|}
I_{\nu}^\mathcal{D} (d-s) t_N^{1+s/d}+o(t_N^{1+s/d}),$$
where 
$$ I_{\nu}^\mathcal{D} (t)=\int_{\mathcal{D}\times \mathcal{D}} 
\frac{1}{|x-y|^t} d\nu(x)d\nu(y),\;\;0<t<d.$$
The measure $\nu$ is a multiple of the Lebesgue measure such that if $\Omega^*$ 
is a fundamental domain for $\Lambda^*$ then $\nu(\Omega^*)=1$.
\end{proposition}

\begin{proof}
From Theorem \ref{explicit} we get that
\begin{align*}
\mathbb{E}_{\omega\in (\R^d)^{t_N}} & (E_{s,\Lambda}(\omega))
=\frac{2\pi^{d/2}|\Lambda|^{-1}}{\Gamma\left(  \frac{s}{2} \right)(d-s) } t_N^2
\\
&
-
\frac{\pi^{s-\frac{d}{2}} \Gamma\left(  \frac{d-s}{2} \right) }{ \Gamma\left(  
\frac{s}{2} \right) |\Lambda|}\sum_{\substack{w,w' \in \Lambda^* \\ w\neq w'}} 
\frac{\kappa_N(w)\kappa_N(w')}{|w-w'|^{d-s}}+o(t_N^{1+\frac{d}{s}}). 
\end{align*}

Observe that
$$\frac{t_N}{N}=\frac{1}{N}\sum_{w\in \Lambda^*} \kappa_N(w)=\frac{\#(\Z^d \cap 
N^{1/d} A^{t}\mathcal{D} )}{N}\longrightarrow |A^{t}\mathcal{D}|=|\Lambda| 
|\mathcal{D}|,\;N\to \infty.$$

For the other term,
$$\sum_{\substack{w,w' \in \Lambda^* \\ w\neq w'}} 
\frac{\kappa_N(w)\kappa_N(w')}{|w-w'|^{d-s}}=
\sum_{\substack{w,w' \in \Lambda^*,w\neq w' \\ w,w' \in N^{1/d}\mathcal{D} }} 
\frac{1}{|w-w'|^{d-s}}$$
$$
=N^{1+s/d}\frac{1}{N^2}
\sum_{\substack{z\neq z' \in \Z^d \\ z,z' \in N^{1/d} A^{t} \mathcal{D} }}
\frac{1}{\left| A^{-t} ( N^{-1/d}z-N^{-1/d}z' )   \right|^{d-s}},
$$
and 
$$\lim_{N\to \infty} \frac{1}{N^2}
\sum_{\substack{z\neq z' \in \Z^d \\ z,z' \in N^{1/d} A^{t} \mathcal{D} }}
\frac{1}{\left| A^{-t} ( N^{-1/d}z-N^{-1/d}z' )   \right|^{d-s}}=
\int_{A^t\mathcal{D}\times A^t\mathcal{D}}\frac{1}{|A^{-t}(x-y)|^{d-s}}dxdy$$

$$=
|\det A|^2 \int_{\mathcal{D}\times \mathcal{D}}\frac{1}{|x-y|^{d-s}}dxdy=
\int_{\mathcal{D}\times \mathcal{D}}\frac{1}{|x-y|^{d-s}}d\nu(x)d\nu(y)=
 I_{\nu}^\mathcal{D} (d-s).$$
\end{proof}

A natural question is now, given a fixed lattice $\Lambda,$ to find the optimal 
$\mathcal{D}\subset \R^d$ 
in the definition of $\kappa^{\mathcal{D},\Lambda}$.
By the result above, all we have to check is what is the domain giving the 
larger potential $I_{\nu}^\mathcal{D}$. 
The result follows from Riesz rearrangement inequality \cite[p. 87]{LL01}.

\begin{theorem}[Riesz rearrangement inequality]
Given $f,g,H$ nonnegative functions in $\R^d$ with $h(x)=H(|x|)$ symmetrically 
decreasing. Then 
$$\int_{\R^d}\int_{\R^d} f(x)g(y)H(|x-y|)dxdy\le \int_{\R^d}\int_{\R^d} 
f^*(x)g^*(y)H(|x-y|)dxdy,$$
where $f^*,g^*$ are the symmetric decreasing rearrangements of $f$ and $g$.
\end{theorem}

We apply the result for $f=g=\chi_{\mathcal{D}}$ getting 
$f^*=g^*=\chi_{\mathcal{D}^*}$ where
$\mathcal{D}^*$ is the open ball centered at the origin with 
$|\mathcal{D}^*|=|\mathcal{D}|$. If 
$\omega_{d-1}=d \frac{\pi^{d/2}}{\Gamma\left(\frac{d}{2}+1 \right)}$ is the 
surface measure of the unit ball $\mathbb{S}^{d-1}$
in $\R^d$ then
$$\mathcal{D}^*=\left\{ x\in \R^d \;:\; \omega_{d-1}|x|^d<d|\mathcal{D}| 
\right\}=B(0,r_d),\;\mbox{with}\;\;
r_d=\left( \frac{d|\mathcal{D}|}{\omega_{d-1}}\right)^{1/d}.$$
%
%
%
%

\subsection{An upper bound for the minimal periodic Riesz $s$-energy}

It is clear that 
$$\mathcal{E}_{s,\Lambda}(N)\le \mathbb{E}_{\omega\in 
(\R^d)^{N}}(E_{s,\Lambda}(\omega)),$$ 
and therefore from the results above we get that $C_{s,d}$  in 
(\ref{asymptotics}) satisfies, for $0<s<d,$
$$C_{s,d} \le -\frac{\pi^{s-\frac{d}{2}} \Gamma\left(  \frac{d-s}{2} \right) }{ 
\Gamma\left(  \frac{s}{2} \right) |\Lambda|^{1-\frac{s}{d}}}
I_{\nu}^\mathcal{D} (d-s),$$
where $\mathcal{D}$ is any bounded domain such that $|\mathcal{D}||\Lambda|=1$. 
By the discussion above the best choice is to
take the set $\mathcal{D}$ to be a ball.

\begin{proposition}				\label{proposition-bound}
Let $\Lambda\subset \R^d$ be a lattice and 
$$\mathcal{D}=B(0,r_d),\;\mbox{with}\;\;
r_d=\left( \frac{d|\Lambda|^{-1}}{\omega_{d-1}}\right)^{1/d}.$$ 
Then $|\mathcal{D}||\Lambda|=1$ and
$$\frac{\pi^{s-\frac{d}{2}} \Gamma\left(  \frac{d-s}{2} \right) }{ \Gamma\left( 
 \frac{s}{2} \right) |\Lambda|^{1-\frac{s}{d}}}
I_{\nu}^\mathcal{D} (d-s)=d\left[ 2\pi \left( 
\frac{d}{\omega_{d-1}}\right)^{1/d}\right]^s \int_0^\infty 
\frac{J_{d/2}^2(t)}{t^{1+s}}dt,$$
where according to \cite[p.47, (4)]{EOT54}
$$\int_0^\infty \frac{J_{d/2}^2(t)}{t^{1+s}}dt
=\frac{ \Gamma\left( \frac{d-s}{2} \right) }{2^{d+1}\Gamma\left( \frac{d}{2}+1 
\right)\Gamma\left( \frac{s}{2}+1 \right)}
{}_{2} F_1\left( \frac{d-s}{2},\frac{d+1}{2};d+1;1\right)$$
$$=
\frac{ \Gamma\left( \frac{d-s}{2} \right) \Gamma\left( d+1 \right) \Gamma\left( 
\frac{s+1}{2} \right) }{2^{d+1}\Gamma\left( \frac{d}{2}+1 \right)\Gamma\left( 
\frac{s}{2}+1 \right)
\Gamma\left( \frac{d+s}{2}+1 \right)\Gamma\left( \frac{d+1}{2} \right)}.
$$
\end{proposition}

\begin{proof}
The proof is an easy computation. We take the normalization of the Fourier 
transform
$$\hat{f}(\xi)=\int_{\R^d} f(x) e^{-2\pi i \langle x,\xi \rangle} dx,$$
and then, in distributional sense,
if $f_s(x)=|x|^{-s},$
$$\hat{f_s}(\xi)=\frac{\pi^{s-\frac{d}{2}} \Gamma\left(  \frac{d-s}{2} \right) 
}{ \Gamma\left(  \frac{s}{2} \right)}\frac{1}{|\xi|^{d-s}}.$$

For any $K\subset \R^d$ we have 
$$\int_{K\times K} \frac{1}{|x-y|^s}dx dy=\int_{K} (f_s\ast \chi_K)(y) 
dy=(f_s\ast \chi_K,\chi_K)$$
$$=
(\hat{f_s} \widehat{\chi_K},\widehat{\chi_K})=\frac{\pi^{s-\frac{d}{2}} 
\Gamma\left(  \frac{d-s}{2} \right) }{ \Gamma\left(  \frac{s}{2} \right)}
\int_{\R^d} \frac{|\widehat{\chi_K}(\xi)|^2}{|\xi|^{d-s}} d\xi,$$
and the result follows from
$$\widehat{\chi_{B(0,r)}}(\xi)=r^{d/2} \frac{J_{d/2}(2\pi 
r|\xi|)}{|\xi|^{d/2}}.$$ 
\end{proof}

\begin{remark}			\label{remark2}
 The function
 $$A_{s,d}=-d\left[ 2\pi \left( 
\frac{d}{\omega_{d-1}}\right)^{1/d}\right]^s \frac{ \Gamma\left( \frac{d-s}{2} \right) \Gamma\left( d+1 \right) \Gamma\left( 
\frac{s+1}{2} \right) }{2^{d+1}\Gamma\left( \frac{d}{2}+1 \right)\Gamma\left( 
\frac{s}{2}+1 \right)
\Gamma\left( \frac{d+s}{2}+1 \right)\Gamma\left( \frac{d+1}{2} \right)}
,$$
 is such that $A_{0,s}=-1$ and the residue in $d$ is
$$\lim_{s\to d} (s-d)A_{s,d}=\omega_{d-1},$$
as the Epstein zeta functions, and is negative for $0<s<d$. Recall that it has been conjectured (see \cite{BHS12}) that $C_{s,d}$ 
is  given by $\zeta_{\Lambda_d}(s)$ for $d=2,4,8,24,$
where $\Lambda_d$ denotes (respectively) the hexagonal lattice, the $D_4$  lattice, the $E_8$ lattice and the Leech lattice 
(scaled to have $|\Lambda_d|=1$). 
The value of $A_{s,d}$ is bigger (i.e. worst) than these conjectured values for $C_{s,d}$ with
$d=2,4,8,24,$ see figure \ref{fig:1}. However, it is observed in \cite{SS06} that no 
negative bound was known for $C_{s,d}$ for large dimensions. Indeed,
it was proved in
\cite{Ter80} (see also \cite[4.4.4.]{Ter88}) that for $0<\delta<1$ and $d$ sufficiently large then 
$$\zeta_\Lambda(\delta d)>0,$$ for every lattice $\Lambda,$ such that the shortest non-zero vector in $\Lambda$ satisfies
$$m(\Lambda)\le \delta \sqrt{\frac{d}{\pi e}},$$
and this last condition is satisfied by all explicitly known lattices in large dimensions.
\end{remark}

\begin{figure}
\begin{center}
\includegraphics[width=.4\textwidth]{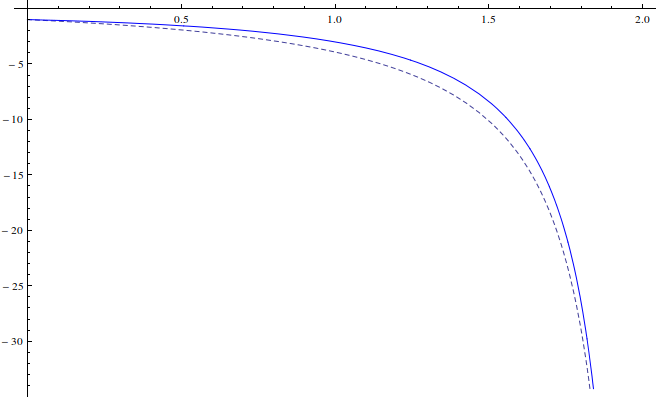}
\hskip 0.4cm
\includegraphics[width=.4\textwidth]{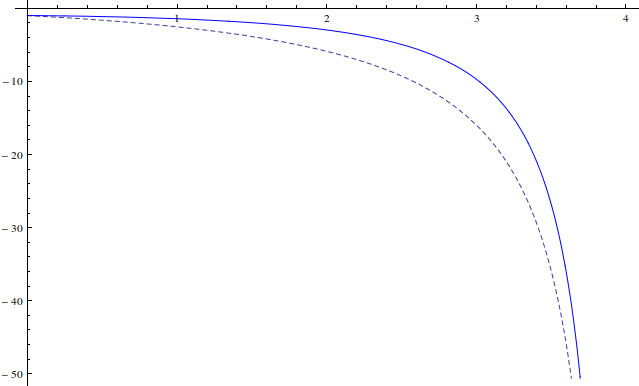}
\vskip 0.4cm
\includegraphics[width=.4\textwidth]{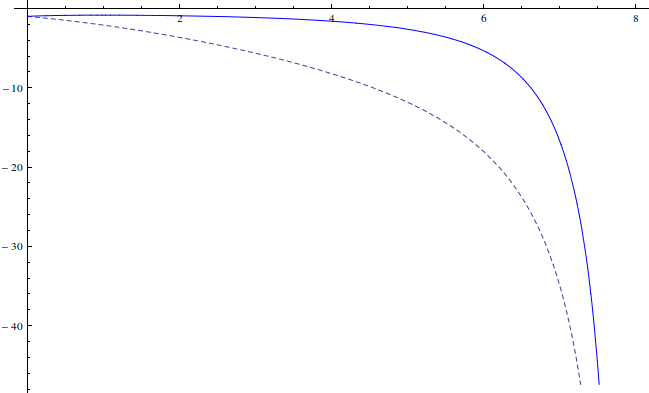}
\hskip 0.4cm
\includegraphics[width=.4\textwidth]{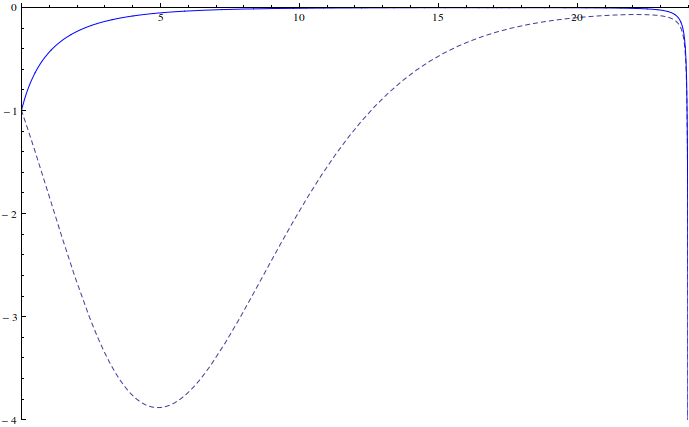}
\end{center}
\caption{$A_{s,d}$ dashed and $\zeta_{\Lambda_{d}}(s)$ solid for $d=2,4,8,24$}
\label{fig:1}
\end{figure}

\end{document}